\documentclass{article}

\usepackage{amsthm,amssymb,amsmath}
\usepackage{a4wide}

\newtheorem{theorem}{Theorem}
\newtheorem{corollary}[theorem]{Corollary}
\newtheorem{proposition}[theorem]{Proposition}
\newtheorem{definition}[theorem]{Definition}

\newtheorem{remark}[theorem]{Remark}
\newtheorem{lemma}[theorem]{Lemma}

\begin{document}

\title{Existence and Uniqueness of Solution for a Fractional
Riemann--Liouville Initial Value Problem on Time Scales\thanks{This
is a preprint of a paper whose final and definite form is in \emph{Journal
of King Saud University (Science)}, DOI: 10.1016/j.jksus.2015.08.001. 
Paper submitted 18/May/2015; revised 31/July/2015; 
accepted for publication 03/Aug/2015.}}

\author{Nadia Benkhettou$^1$\\ \texttt{benkhettou$_{-}$na@yahoo.com}
\and Ahmed Hammoudi$^{2}$\\ \texttt{hymmed@hotmail.com}
\and Delfim F. M. Torres$^3$\thanks{Corresponding author.
Tel: +351 234370668; Fax: +351 234370066;  Email: delfim@ua.pt}\\
\texttt{delfim@ua.pt}}

\date{$^1$Laboratoire de Math\'{e}matiques,
Universit\'{e} de Sidi Bel-Abb\`{e}s,\\
B.P. 89, 22000 Sidi Bel-Abb\`{e}s, Algerie\\[0.3cm]
$^2$Laboratoire de Math\'{e}matiques, Universit\'{e} de Ain T\'{e}mouchentl,\\
B.P. 89, 46000 Ain T\'{e}mouchent, Algerie\\[0.3cm]
$^3$\text{Center for Research and Development in Mathematics and Applications (CIDMA)},\\
Department of Mathematics, University of Aveiro, 3810-193 Aveiro, Portugal}

\maketitle


\begin{abstract}
We introduce the concept of fractional derivative
of Riemann--Liouville on time scales.
Fundamental properties of the new operator are proved,
as well as an existence and uniqueness result for a fractional
initial value problem on an arbitrary time scale.

\bigskip


\bigskip

\end{abstract}


\section{Introduction}

Let $\mathbb{T}$ be a time scale, that is,
a closed subset of $\mathbb{R}$.
We consider the following initial value problem:
\begin{equation}
\label{eq1}
{_{t_{0}}^{\mathbb{T}}D}_{t}^{\alpha} y(t)=f(t,y(t)),
\quad t\in[t_{0}, t_{0}+a]=\mathcal{J}\subseteq\mathbb{T},
\quad 0< \alpha <1,
\end{equation}
\begin{equation}
\label{eq2}
{_{t_{0}}^{\mathbb{T}}I}_{t}^{1-\alpha}y(t_{0})= 0,
\end{equation}
where ${_{t_{0}}^{\mathbb{T}}D}_{t}^{\alpha}$ is the (left)
Riemann--Liouville fractional derivative
operator or order $\alpha$ defined on $\mathbb{T}$,
${_{t_{0}}^{\mathbb{T}}I}_{t}^{1-\alpha}$
the (left) Riemann--Liouville fractional integral
operator or order $1-\alpha$ defined on $\mathbb{T}$,
and function $f:\mathcal{J} \times\mathbb{T}\rightarrow \mathbb{R}$
is a right-dense continuous function. Our main results give necessary
and sufficient conditions for the existence and uniqueness of solution
to problem \eqref{eq1}--\eqref{eq2}.


\section{Preliminaries}
\label{sec:prelim}

In this section, we collect notations, definitions, and results,
which are needed in the sequel. We use $\mathcal{C}(\mathcal{J},\mathbb{R})$
for a Banach space of continuous functions $y$ with the norm
$\|y\|_{\infty}=\sup\left\{|y(t)| : t\in\mathcal{J}\right\}$,
where $\mathcal{J}$ is an interval. A time scale $ \mathbb{T}$ is an arbitrary
nonempty closed subset of $ \mathbb{R}$. The reader interested on the
calculus on time scales is referred to the books \cite{BP,BP1}.
For a survey, see \cite{ABRP}. Any time scale $\mathbb{T}$
is a complete metric space with the distance $d(t,s)=|t-s|$, $t,s\in\mathbb{T}$.
Consequently, according to the well-known theory of general metric spaces,
we have for $\mathbb{T}$ the fundamental concepts such as open balls (intervals),
neighborhoods of points, open sets, closed sets, compact sets, etc.
In particular, for a given number $\delta>0$, the $\delta$-neighborhood
$\mathrm{U}_{\delta}(t)$ of a given point $t\in\mathbb{T}$ is the set of all points
$s\in\mathbb{T}$ such that $d(t,s)< \delta$. We also have, for functions
$f:\mathbb{T}\rightarrow\mathbb{R}$, the concepts of limit,
continuity, and the properties of continuous functions on a general complete metric
space. Roughly speaking, the calculus on time scales begins by introducing
and investigating the concept of derivative for functions
$f :\mathbb{T}\rightarrow\mathbb{R}$. In the definition of derivative,
an important role is played by the so-called jump operators \cite{BP1}.

\begin{definition}
\label{def:jump:oper}
Let $\mathbb{T}$ be a time scale. For $t \in \mathbb{T}$
we define the forward jump operator
$\sigma:\mathbb{T}\rightarrow \mathbb{T}$ by
$\sigma(t):=\inf\{s \in\mathbb{T} : s > t\}$,
and the backward jump operator
$\rho:\mathbb{T}\rightarrow \mathbb{T}$ by
$\rho(t):=\sup\{s \in\mathbb{T} : s < t\}$.
\end{definition}

\begin{remark}
In Definition~\ref{def:jump:oper}, we put $\inf \emptyset =\sup \mathbb{T}$
(i.e., $\sigma(M)= M$ if $\mathbb{T}$ has a maximum $M$)
and $\sup \emptyset =\inf \mathbb{T}$ (i.e., $\rho(m)= m$
if $\mathbb{T}$ has a minimum $m$), where $\emptyset$ denotes
the empty set.
\end{remark}

If $\sigma(t) > t$, then we say that $t$ is right-scattered;
if $\rho(t) < t$, then $t$ is said to be left-scattered.
Points that are simultaneously right-scattered and left-scattered
are called isolated. If $t < \sup\mathbb{T}$ and $\sigma(t) = t$,
then $t$ is called right-dense; if $t >\inf \mathbb{T}$ and $\rho(t)= t$,
then $t$ is called left-dense. The graininess function
$\mu :\mathbb{T}\rightarrow [0,\infty)$ is defined by
$\mu(t) :=\sigma(t) - t$.

The derivative makes use of the set $\mathbb{T}^{\kappa}$,
which is derived from the time scale $\mathbb{T}$ as follows:
if $\mathbb{T}$ has a left-scattered maximum $M$, then
$\mathbb{T}^{\kappa}:=\mathbb{T} \setminus \{M\}$;
otherwise, $\mathbb{T}^{\kappa}:=\mathbb{T}$.

\begin{definition}[Delta derivative \cite{AB}]
Assume $f:\mathbb{T}\rightarrow \mathbb{R}$ and let
$t\in \mathbb{T}^{\kappa}$. We define
$$
f^{\Delta}(t):=\lim_{s\rightarrow t}\frac{f(\sigma(s))-f(t)}{\sigma(s)-t},
\quad t \neq \sigma(s),
$$
provided the limit exists. We call $f^{\Delta}(t)$ the delta derivative
(or Hilger derivative) of $f$ at $t$. Moreover, we say that $f$
is delta differentiable on $\mathbb{T}^{\kappa}$ provided
$f^{\Delta}(t)$ exists for all $t\in \mathbb{T}^{\kappa}$. The function
$f^{\Delta}:\mathbb{T}^{\kappa}\rightarrow \mathbb{R}$ is then called
the (delta) derivative of $f$ on $\mathbb{T}^{\kappa}$.
\end{definition}

\begin{definition}
A function $f:\mathbb{T}\rightarrow \mathbb{R}$ is called rd-continuous provided
it is continuous at right-dense points in $\mathbb{T} $ and its left-sided limits
exist (finite) at left-dense points in $\mathbb{T}$. The set of rd-continuous
functions $f:\mathbb{T}\rightarrow \mathbb{R}$ is denoted by $\mathcal{C}_{rd}$.
Similarly, a function $f:\mathbb{T}\rightarrow \mathbb{R}$ is called ld-continuous provided
it is continuous at left-dense points in $\mathbb{T} $ and its right-sided limits
exist (finite) at right-dense points in $\mathbb{T}$. The set of ld-continuous
functions $f:\mathbb{T}\rightarrow \mathbb{R}$ is denoted by $\mathcal{C}_{ld}$.
\end{definition}

\begin{definition}
Let $[a,b]$ denote a closed bounded interval in $\mathbb{T}$.
A function $F: [a,b]\rightarrow \mathbb{R}$ is called a delta antiderivative
of function $f: [a,b)\rightarrow \mathbb{R}$ provided $F$ is continuous
on $[a,b]$, delta differentiable on $[a,b)$, and $ F^{\Delta}(t)=f(t)$
for all $t\in [a,b)$. Then, we define the $\Delta$-integral of $f$
from $a$ to $b$ by
$$
\int_{a}^{b}f(t)\Delta t := F(b)-F(a).
$$
\end{definition}

\begin{proposition}[See \cite{AJ}]
\label{P1}
Suppose $\mathbb{T}$ is a time scale and
$f$ is an increasing continuous function
on the time-scale interval $[a,b]$.
If $F$ is the extension of $f$ to the real interval $[a,b]$ given by
\begin{equation*}
F(s) :=
\begin{cases}
f(s) & \textrm{ if } s \in\mathbb{T} , \\
f(t) & \textrm{ if } s \in (t,\sigma(t))\notin\mathbb{T},
\end{cases}
\end{equation*}
then
$$
\int_{a}^{b} f(t) \Delta t
\leq \int_{a}^{b} F(t)dt.
$$
\end{proposition}

We also make use of the classical gamma and beta functions.

\begin{definition}[Gamma function]
For complex numbers with a positive real part,
the gamma function $\Gamma(t)$ is defined
by the following convergent improper integral:
$$
\Gamma(t) := \int_{0}^{\infty} x^{t-1} e^{-x}dx.
$$
\end{definition}

\begin{definition}[Beta function]
The beta function, also called the Euler integral of first kind,
is the special function $\mathrm{B}(x,y)$ defined by
$$
\mathrm{B}(x,y) := \int _{0}^{1}t^{x-1}(1-t)^{y-1}dt,
\quad x>0, \quad y>0.
$$
\end{definition}

\begin{remark}
The gamma function satisfies the following useful property:
$\Gamma(t+1)=t \Gamma(t)$.
The beta function can be expressed through the gamma function by
$\mathrm{B}(x,y)=\frac{\Gamma(x) \Gamma(y)}{\Gamma(x+y)}$.
\end{remark}


\section{Main Results}

We introduce a new notion of fractional derivative on time scales.
Before that, we define the fractional integral on a time scale $\mathbb{T}$.
This is in contrast with \cite{BBT,MyID:320,MyID:324}, where first a notion
of fractional differentiation on time scales is introduced and only after that,
with the help of such concept, the fraction integral is defined.

\begin{definition}[Fractional integral on time scales]
\label{def:FI}
Suppose $\mathbb{T}$ is a time scale, $[a,b]$ is an interval of $\mathbb{T}$,
and $h$ is an integrable function on $[a,b]$. Let $0 < \alpha <1$.
Then the (left) fractional integral of order $\alpha$ of $h$ is defined by
$$
{_{a}^{\mathbb{T}}I}_{t}^{\alpha}h(t)
:= \int_{a}^{t} \frac{(t-s)^{\alpha-1}}{\Gamma(\alpha)}h(s)\Delta s,
$$
where $\Gamma$ is the gamma function.
\end{definition}

\begin{definition}[Riemann--Liouville fractional derivative on time scales]
\label{d1}
Let $\mathbb{T}$ be a time scale, $t\in\mathbb{T}$, $0 < \alpha <1$,
and $h:\mathbb{T}\rightarrow \mathbb{R}$. The (left) Riemann--Liouville
fractional derivative of order $\alpha$ of $h$ is defined by
\begin{equation}
\label{eq3}
{_{a}^{\mathbb{T}}D}_{t}^{\alpha}h(t)
:=\frac{1}{\Gamma(1-\alpha)}\left(\int_{a}^{t}
(t-s)^{-\alpha}h(s)\Delta s\right)^{\Delta}.
\end{equation}
\end{definition}

\begin{remark}
If $\mathbb{T}=\mathbb{R}$, then Definition~\ref{d1}
gives the classical (left) Riemann--Liouville fractional derivative
\cite{book:Podlubny}. For different extensions of the fractional derivative
to time scales, using the Caputo approach instead of the
Riemann--Liouville, see \cite{AJ,MR2800417}.
For local approaches to fractional calculus on time scales
we refer the reader to \cite{BBT,MyID:320,MyID:324}.
Here we are only considering left operators.
The corresponding right operators are easily obtained by changing the limits
of integration in Definitions~\ref{def:FI} and \ref{d1} from
$a$ to $t$ (left of $t$) into $t$ to $b$ (right of $t$),
as done in the classical fractional calculus \cite{book:Podlubny}.
Here we restrict ourselves to the delta approach to time scales.
Analogous definitions are, however, trivially obtained for the nabla
approach to time scales by using the duality theory of \cite{MyID:307}.
\end{remark}

Along the work, we consider the order $\alpha$ of the fractional derivatives
in the real interval $(0,1)$. We can, however, easily generalize our definition
of fractional derivative to any positive real $\alpha$. Indeed,
let $\alpha \in \mathbb{R}^+ \setminus \mathbb{N}$.
Then there exists $\beta \in (0,1)$
such that $\alpha=\lfloor\alpha\rfloor+\beta$,
where $\lfloor \alpha \rfloor$ is the integer part of $\alpha$, and we can set
$$
{_{a}^{\mathbb{T}}D}_{t}^{\alpha}h
:=
{_{a}^{\mathbb{T}}D}_{t}^{\beta}
h^{\Delta^{\lfloor\alpha\rfloor}}.
$$
Fractional operators of negative order are defined as follows.

\begin{definition}
\label{d2}
If $-1 <\alpha <0$, then the (Riemann--Liouville) fractional derivative
of order $\alpha$ is the fractional integral of order $-\alpha$, that is,
$$
{_{a}^{\mathbb{T}}D}_{t}^{\alpha}
:= {_{a}^{\mathbb{T}}I}_{t}^{-\alpha}.
$$
\end{definition}

\begin{definition}
\label{d3}
If $-1 <\alpha <0$, then the fractional integral of order $\alpha$
is the fractional derivative of order $-\alpha$, that is,
$$
{_{a}^{\mathbb{T}}I}_{t}^{\alpha}
:= {_{a}^{\mathbb{T}}D}_{t}^{-\alpha}.
$$
\end{definition}


\subsection{Properties of the time-scale fractional operators}
\label{sec:properties}

In this section we prove some fundamental properties
of the fractional operators on time scales.

\begin{proposition}
\label{P2}
Let $\mathbb{T}$ be a time scale with derivative $\Delta$,
and $0 < \alpha <1$. Then,
$$
{_{a}^{\mathbb{T}}D}_{t}^{\alpha}
= \Delta \circ {_{a}^{\mathbb{T}}I_{t}^{1-\alpha}}.
$$
\end{proposition}

\begin{proof}
Let $h:\mathbb{T}\rightarrow\mathbb{R}$. From \eqref{eq3} we have
\begin{equation*}
\begin{split}
{_{a}^{\mathbb{T}}D}_{t}^{\alpha}h(t)
&=\frac{1}{\Gamma(1-\alpha)}\left(\int_{a}^{t}
(t-s)^{-\alpha}h(s)\Delta s\right)^{\Delta}\\
&=\left(_{a}^{\mathbb{T}}I_{t}^{1-\alpha}h(t)\right)^{\Delta}
=\left(\Delta \circ {_{a}^{\mathbb{T}}I_{t}^{1-\alpha}}\right) h(t).
\end{split}
\end{equation*}
The proof is complete.
\end{proof}

\begin{proposition}
\label{P3}
For any function $h$ integrable on $[a,b]$, the
Riemann--Liouville $\Delta$-fractional integral satisfies
${_{a}^{\mathbb{T}}I}_{t}^{\alpha}
\circ {_{a}^{\mathbb{T}}I_{t}^{\beta}}
= {_{a}^{\mathbb{T}}I_{t}^{\alpha+\beta}}$
for $\alpha> 0$ and $\beta > 0$.
\end{proposition}

\begin{proof}
By definition,
\begin{equation*}
\begin{split}
\bigg({_{a}^{\mathbb{T}}I}_{t}^{\alpha}
&\circ {_{a}^{\mathbb{T}}I_{t}^{\beta}}\bigg)(h(t))
={_{a}^{\mathbb{T}}I}_{t}^{\alpha}\left(
{_{a}^{\mathbb{T}}I}_{t}^{\beta}(h(t))\right)\\
&=\frac{1}{\Gamma(\alpha)}\int_{a}^{t}
(t-s)^{\alpha-1}\left({_{a}^{\mathbb{T}}I}_{t}^{\beta}(h(s))\right)\Delta s\\
&=\frac{1}{\Gamma(\alpha)}\int_{a}^{t}
\left((t-s)^{\alpha-1}\frac{1}{\Gamma(\beta)}
\int_{a}^{s}(s-u)^{\beta-1}h(u)\Delta u\right)\Delta s\\
&=\frac{1}{\Gamma(\alpha)\Gamma(\beta)}\int_{a}^{t}\int_{a}^{s}
(t-s)^{\alpha-1}(s-u)^{\beta-1}h(u)\Delta u\Delta s\\
&=\frac{1}{\Gamma(\alpha)\Gamma(\beta)}\int_{a}^{t}
\left[\int_{a}^{s}(t-s)^{\alpha-1}(s-u)^{\beta-1}h(u)\Delta u
+\int_{s}^{t}(t-s)^{\alpha-1}(s-u)^{\beta-1}h(u)\Delta\right]\Delta s\\
&=\frac{1}{\Gamma(\alpha)\Gamma(\beta)}\int_{a}^{t}
\left[\int_{a}^{t}(t-s)^{\alpha-1}(s-u)^{\beta-1}h(u)\Delta u\right]\Delta s.
\end{split}
\end{equation*}
From Fubini's theorem, we interchange the order of integration to obtain
\begin{equation*}
\begin{split}
\left({_{a}^{\mathbb{T}}I}_{t}^{\alpha}
\circ {{_{a}^{\mathbb{T}}I}_{t}^{\beta}}\right)(h(t))
&=\frac{1}{\Gamma(\alpha)\Gamma(\beta)}\int_{a}^{t}
\left[\int_{a}^{t}(t-s)^{\alpha-1}(s-u)^{\beta-1}h(u)\Delta s\right]\Delta u\\
&=\frac{1}{\Gamma(\alpha)\Gamma(\beta)}\int_{a}^{t}
\left[\int_{a}^{t}(t-s)^{\alpha-1}(s-u)^{\beta-1}\Delta s\right]h(u)\Delta u\\
&=\frac{1}{\Gamma(\alpha)\Gamma(\beta)}\int_{a}^{t}
\left[\int_{u}^{t}(t-s)^{\alpha-1}(s-u)^{\beta-1}\Delta s\right]h(u)\Delta u.
\end{split}
\end{equation*}
By setting $s=u+r(t-u)$, $r\in\mathbb{R}$, we obtain that
\begin{equation*}
\begin{split}
& \left({_{a}^{\mathbb{T}}I}_{t}^{\alpha}
\circ {_{a}^{\mathbb{T}}I_{t}^{\beta}}\right)(h(t))\\
&\quad=\frac{1}{\Gamma(\alpha)\Gamma(\beta)}\int_{a}^{t}
\left[\int_{0}^{1}(1-r)^{\alpha-1}(t-u)^{\alpha-1}
r^{\beta-1}(t-u)^{\beta-1}(t-u)dr \right]h(u)\Delta u\\
&\quad=\frac{1}{\Gamma(\alpha)\Gamma(\beta)}\int_{0}^{1}
(1-r)^{\alpha-1}r^{\beta-1}d r\int_{a}^{t}(t-u)^{\alpha+\beta-1}h(u)\Delta u\\
&\quad=\frac{B(\alpha,\beta)}{\Gamma(\alpha)\Gamma(\beta)}
\int_{a}^{t}(t-u)^{\alpha+\beta-1}h(u)\Delta u
=\frac{1}{\Gamma(\alpha+\beta)}
\int_{a}^{t}(t-u)^{\alpha+\beta-1}h(u)\Delta u\\
&\quad= _{a}^{\mathbb{T}}I_{t}^{\alpha+\beta}h(t).
\end{split}
\end{equation*}
The proof is complete.
\end{proof}

\begin{proposition}
\label{P4}
For any function $h$ integrable on $[a,b]$ one has
${_{a}^{\mathbb{T}}D}_{t}^{\alpha}
\circ {_{a}^{\mathbb{T}}I_{t}^{\alpha}}h = h$.
\end{proposition}

\begin{proof}
By Propositions~\ref{P2} and \ref{P3}, we have
\begin{equation*}
{_{a}^{\mathbb{T}}D}_{t}^{\alpha}
\circ {_{a}^{\mathbb{T}}I_{t}^{\alpha}}h(t)
=\left[{_{a}^{\mathbb{T}}I}_{t}^{1-\alpha}\left(
{_{a}^{\mathbb{T}}I}_{t}^{\alpha}(h(t))\right)\right]^{\Delta}
=\left[{_{a}^{\mathbb{T}}I}_{t}h(t)\right]^{\Delta}
=h(t).
\end{equation*}
The proof is complete.
\end{proof}

\begin{corollary}
For $0<\alpha<1$, we have
${_{a}^{\mathbb{T}}D}_{t}^{\alpha}
\circ {_{a}^{\mathbb{T}}D_{t}^{-\alpha}}
=Id$
and
${_{a}^{\mathbb{T}}I}_{t}^{-\alpha}
\circ {_{a}^{\mathbb{T}}I_{t}^{\alpha}}=Id$,
where $Id$ denotes the identity operator.
\end{corollary}

\begin{proof}
From Definition~\ref{d3} and Proposition~\ref{P4}, we have that
${_{a}^{\mathbb{T}}D}_{t}^{\alpha}
\circ {_{a}^{\mathbb{T}}D_{t}^{-\alpha}}
= {_{a}^{\mathbb{T}}D}_{t}^{\alpha}
\circ {{_{a}^{\mathbb{T}}I}_{t}^{\alpha}}
= Id$;
from Definition~\ref{d2} and Proposition~\ref{P4}, we have that
${_{a}^{\mathbb{T}}I}_{t}^{-\alpha}
\circ {{_{a}^{\mathbb{T}}I}_{t}^{\alpha}}
= {{_{a}^{\mathbb{T}}D}_{t}^{\alpha}}
\circ {{_{a}^{\mathbb{T}}I}_{t}^{\alpha}}
= Id$.
\end{proof}

\begin{definition}
For $\alpha>0$, let ${_{a}^{\mathbb{T}}I}_{t}^{\alpha}([a,b])$ denote the space
of functions that can be represented by the Riemann--Liouville $\Delta$ integral
of order $\alpha$ of some $\mathcal{C}([a,b])$-function.
\end{definition}

\begin{theorem}
\label{th1}
Let $f\in \mathcal{C}([a,b])$ and $\alpha>0$.
In order that $f \in {_{a}^{\mathbb{T}}}I_{t}^{\alpha}([a,b])$,
it is necessary and sufficient that
\begin{equation}
\label{e4}
{_{a}^{\mathbb{T}}I}_{t}^{1-\alpha}f \in C^1([a,b])
\end{equation}
and
\begin{equation}
\label{e5}
\left.\left(_{a}^{\mathbb{T}}I_{t}^{1-\alpha}f(t)\right)\right|_{t=a}=0.
\end{equation}
\end{theorem}

\begin{proof}
Assume $f \in {_{a}^{\mathbb{T}}I}_{t}^{\alpha}([a,b])$,
$f(t) = {_{a}^{\mathbb{T}}I}_{t}^{\alpha}g(t)$
for some $g\in\mathcal{C}([a,b])$, and
$$
{_{a}^{\mathbb{T}}I}_{t}^{1-\alpha}(f(t))
= {_{a}^{\mathbb{T}}I}_{t}^{1-\alpha}\left(
{_{a}^{\mathbb{T}}I}_{t}^{\alpha}g(t)\right).
$$
From Proposition~\ref{P3}, we have
$$
{_{a}^{\mathbb{T}}I}_{t}^{1-\alpha}(f(t))
={_{a}^{\mathbb{T}}I}_{t}g(t)
=\int_{a}^{t}g(s)\Delta s.
$$
Therefore,
${_{a}^{\mathbb{T}}I}_{t}^{1-\alpha} f \in \mathcal{C}([a,b])$
and
$$
\left.\left({_{a}^{\mathbb{T}}I}_{t}^{1-\alpha}f(t)\right)\right|_{t=a}
=\int_{a}^{t}g(s)\Delta s=0.
$$
Conversely, assume that $f\in \mathcal{C}([a,b])$
satisfies \eqref{e4} and \eqref{e5}. Then, by Taylor's formula applied to
function ${_{a}^{\mathbb{T}}I}_{t}^{1-\alpha} f$, one has
$$
{_{a}^{\mathbb{T}}I}_{t}^{1-\alpha}f(t)
=\int_{a}^{t}\frac{\Delta}{\Delta s}
{_{a}^{\mathbb{T}}I}_{t}^{1-\alpha}f(s)\Delta s,
\quad \forall t\in[a,b].
$$
Let $\varphi(t) := \frac{\Delta}{\Delta s}
{_{a}^{\mathbb{T}}I}_{t}^{1-\alpha}f(t)$.
Note that $\varphi \in \mathcal{C}([a,b])$ by \eqref{e4}.
Now, by Proposition~\ref{P3}, we have
$$
{_{a}^{\mathbb{T}}I}_{t}^{1-\alpha}(f(t))
= {_{a}^{\mathbb{T}}I}_{t}^{1}\varphi(t)
= {_{a}^{\mathbb{T}}I}_{t}^{1-\alpha}\left[
{_{a}^{\mathbb{T}}I}_{t}^{\alpha}(\varphi(t))\right]
$$
and thus
$$
{_{a}^{\mathbb{T}}I}_{t}^{1-\alpha}(f(t))
- {_{a}^{\mathbb{T}}I}_{t}^{1-\alpha}\left[
{_{a}^{\mathbb{T}}I}_{t}^{\alpha}(\varphi(t))\right]
\equiv 0.
$$
Then,
$$
{_{a}^{\mathbb{T}}I}_{t}^{1-\alpha}\left[
f - {_{a}^{\mathbb{T}}I}_{t}^{\alpha}(\varphi(t))\right]
\equiv 0.
$$
From the uniqueness of solution to Abel's integral
equation \cite{MyID:310}, this implies that
$f - {_{a}^{\mathbb{T}}I}_{t}^{\alpha}\varphi\equiv 0$.
Thus, $f = {_{a}^{\mathbb{T}}I}_{t}^{\alpha}\varphi$
and $f \in {_{a}^{\mathbb{T}}I}_{t}^{\alpha}[a,b]$.
\end{proof}

\begin{theorem}
\label{th2}
Let $\alpha > 0$ and $f\in \mathcal{C}([a,b])$ satisfy
the condition in Theorem~\ref{th1}. Then,
$$
\left({_{a}^{\mathbb{T}}I}_{t}^{\alpha}
\circ {_{a}^{\mathbb{T}}D}_{t}^{\alpha}\right)(f) = f.
$$
\end{theorem}

\begin{proof}
By Theorem~\ref{th1} and Proposition~\ref{P3}, we have:
\begin{equation*}
{_{a}^{\mathbb{T}}I}_{t}^{\alpha} \circ {_{a}^{\mathbb{T}}D_{t}^{\alpha}}f(t)
={_{a}^{\mathbb{T}}I}_{t}^{\alpha}
\circ {{_{a}^{\mathbb{T}}D}_{t}^{\alpha}}\left(
{_{a}^{\mathbb{T}}I}_{t}^{\alpha}\varphi(t)\right)
= {_{a}^{\mathbb{T}}I}_{t}^{\alpha}\varphi(t)
=f(t).
\end{equation*}
The proof is complete.
\end{proof}


\subsection{Existence of Solutions to Fractional IVPs on Time Scales}
\label{sec:existence:sol}

In this section we prove existence of solution to the
fractional order initial value problem \eqref{eq1}--\eqref{eq2}
defined on a time scale. For this, let $\mathbb{T}$ be a time scale and
$\mathcal{J}=[t_{0},t_{0}+a]\subset \mathbb{T}$.
Then the function $y\in\mathcal{C}(\mathcal{J},\mathbb{R})$ is a solution
of problem \eqref{eq1}--\eqref{eq2} if
$$
{_{t_{0}}^{\mathbb{T}}D}_{t}^{\alpha}y(t)= f(t,y)
\, \text{ on } \mathcal{J},
$$
$$
{_{t_{0}}^{\mathbb{T}}I}_{t}^{\alpha}y(t_{0})=0.
$$
To establish this solution, we need to prove the following lemma and theorem.

\begin{lemma}
Let $0<\alpha< 1$, $\mathcal{J}\subseteq\mathbb{T}$, and
$f: \mathcal{J}\times\mathbb{R}\rightarrow\mathbb{R}$.
Function $y$ is a solution of problem
\eqref{eq1}--\eqref{eq2} if and only if this function
is a solution of the following integral equation:
$$
y(t)=\frac{1}{\Gamma(\alpha)}\int_{t_{0}}^{t}
(t-s)^{\alpha-1}f(s,y(s))\Delta s.
$$
\end{lemma}

\begin{proof}
By Theorem~\ref{th2},
${_{t_{0}}^{\mathbb{T}}I}_{t}^{\alpha}
\circ \left({_{t_{0}}^{\mathbb{T}}D}_{t}^{\alpha}(y(t))\right)
=y(t)$.
From~\eqref{eq3} we have
$$
y(t)=\frac{1}{\Gamma(\alpha)}\int_{t_{0}}^{t}
(t-s)^{\alpha-1}f(s,y(s))\Delta s.
$$
The proof is complete.
\end{proof}

Our first result is based on the Banach fixed point theorem \cite{C}.

\begin{theorem}
Assume $\mathcal{J}=[t_{0},t_{0}+a]\subseteq\mathbb{T}$.
The initial value problem \eqref{eq1}--\eqref{eq2}
has a unique solution on $\mathcal{J}$ if the function
$f(t,y)$ is a right-dense continuous bounded function
such that there exists $M >0$ for which $|f(t,y(t))|< M $
on $\mathcal{J}$ and the Lipshitz condition
$$
\exists \, L> 0 : \, \forall \, t\in \mathcal{J} \text{ and }
x,y\in\mathbb{R}, \quad \|f(t,x)-f(t,y)\|\leq L\|x-y\|
$$
holds.
\end{theorem}

\begin{proof}
Let $\mathcal{S}$ be the set of rd-continuous functions
on $\mathcal{J}\subseteq \mathbb{T}$. For $y\in\mathcal{S}$, define
$$
\|y\|=\sup_{t\in\mathcal{J}}\|y(t)\|.
$$
It is easy to see that $\mathcal{S}$ is a Banach space with this norm.
The subset of $\mathcal{S}(\rho)$ and the operator $\mathrm{T}$
are defined by
$$
\mathcal{S}(\rho)=\left\{X\in \mathcal{S}:\|X_{s}\|\leq \rho\right\}
$$
and
$$
\mathrm{T}(y)=\frac{1}{\Gamma(\alpha)}
\int_{t_{0}}^{t} (t-s)^{\alpha-1} f(s,y(s))\Delta s.
$$
Then,
\begin{equation*}
|\mathrm{T}(y(t))|
\leq \frac{1}{\Gamma(\alpha)}
\int_{t_{0}}^{t}(t-s)^{\alpha-1}M\Delta s
\leq \frac{M}{\Gamma(\alpha)}\int_{t_{0}}^{t}(t-s)^{\alpha-1}\Delta s.
\end{equation*}
Since $(t-s)^{\alpha-1}$ is an increasing monotone function,
by using Proposition~\ref{P1} we can write that
$$
\int_{t_{0}}^{t} (t-s)^{\alpha-1}\Delta s
\leq \int_{t_{0}}^{t}(t-s)^{\alpha-1} ds.
$$
Consequently,
$$
\left|\mathrm{T}(y(t))\right|
\leq \frac{M}{\Gamma(\alpha)}\int_{t_{0}}^{t}(t-s)^{\alpha-1}ds
\leq \frac{M}{\Gamma(\alpha)}\frac{a^{\alpha}}{\alpha}
= \rho.
$$
By considering
$\rho=\frac{Ma^{\alpha}}{\Gamma(\alpha+1)}$,
we conclude that $\mathrm{T}$ is an operator from
$\mathcal{S}(\rho)$ to $\mathcal{S}(\rho)$. Moreover,
\begin{equation*}
\begin{split}
\|\mathrm{T}(x)-\mathrm{T}(y)\|&\leq\frac{1}{\Gamma(\alpha)}
\int_{t_{0}}^{t}(t-s)^{\alpha-1}|f(s,x(s))-f(s,y(s))|\Delta s\\
&\leq\frac{L\|\|x-y\|_{\infty}}{\Gamma(\alpha)}
\int_{t_{0}}^{t}(t-s)^{\alpha-1}\Delta s\\
&\leq\frac{L\|\|x-y\|_{\infty}}{\Gamma(\alpha)}
\int_{t_{0}}^{t}(t-s)^{\alpha-1}ds\\
&\leq\frac{L\|x-y\|_{\infty}}{\Gamma(\alpha)}\frac{a^{\alpha}}{\alpha}
=\frac{La^{\alpha}}{\Gamma(\alpha+1)}\|x-y\|_{\infty}
\end{split}
\end{equation*}
for $x,y\in\mathcal{S}(\rho)$. If $\frac{La^{\alpha}}{\Gamma(\alpha)} \leq 1$,
then it is a contraction map. This implies the existence and uniqueness
of solution to problem~\eqref{eq1}--\eqref{eq2}.
\end{proof}

\begin{theorem}
Suppose $f :\mathcal{J}\times\mathbb{R}\rightarrow\mathbb{R}$
is a rd-continuous bounded function such that there exists
$M > 0$ with $|f(t,y)|\leq M$ for all $t\in\mathcal{J}$, $y\in\mathbb{R}$.
Then problem~\eqref{eq1}--\eqref{eq2} has a solution on $\mathcal{J}$.
\end{theorem}

\begin{proof}
We use Schauder's fixed point theorem \cite{C}
to prove that $\mathrm{T}$ defined by \eqref{eq3}
has a fixed point. The proof is given in several steps.
\emph{Step 1:} $\mathrm{T}$ is continuous.
Let $y_{n}$ be a sequence such that $y_{n}\rightarrow y$
in $\mathcal{C}(\mathcal{J},\mathbb{R})$. Then, for each $t\in\mathcal{J}$,
\begin{equation*}
\begin{split}
|T(y_{n})(t)&-T(y)(t)|\\
&\leq\frac{1}{\Gamma(\alpha)}\int_{t_{0}}^{t}(t-s)^{\alpha-1}
\left|f(s,y_{n}(s))-f(s,y(s))\right|\Delta s\\
&\leq\frac{1}{\Gamma(\alpha)}\int_{t_{0}}^{t}(t-s)^{\alpha-1}
\sup_{s\in\mathcal{J}}\left|f(s,y_{n}(s))-f(s,y(s))\right|\Delta s\\
&\leq\frac{\left\|f(\cdot,y_{n}(\cdot))
-f(\cdot,y(\cdot))\right\|_{\infty}}{\Gamma(\alpha)}
\int_{t_{0}}^{t}(t-s)^{\alpha-1}\Delta s\\
&\leq\frac{\|f(\cdot,y_{n}(\cdot))
-f(\cdot,y(\cdot))\|_{\infty}}{\Gamma(\alpha)}
\int_{t_{0}}^{t}(t-s)^{\alpha-1}d s\\
&\leq\frac{\|f(\cdot,y_{n}(\cdot))
-f(\cdot,y(\cdot))\|_{\infty}}{\Gamma(\alpha)}
\frac{a^{\alpha}}{\alpha}\\
&\leq\frac{a^{\alpha}\left\|f(\cdot,y_{n}(\cdot))
-f(\cdot,y(\cdot))\right\|_{\infty}}{\Gamma(\alpha+1)}.
\end{split}
\end{equation*}
Since $f$ is a continuous function, we have
$$
\left|T(y_{n})(t)-T(y)(t)\right|_{\infty}
\leq\frac{a^{\alpha}}{\Gamma(\alpha+1)}\left\|f(\cdot,y_{n}(\cdot))
-f(\cdot,y(\cdot))\right\|_{\infty}
\rightarrow 0 \ \text{ as } \ n\rightarrow \infty.
$$
\emph{Step 2:}
the map $\mathrm{T}$ sends bounded sets into bounded set
in $\mathcal{C}(\mathcal{J},\mathbb{R})$. Indeed, it is enough
to show that for any $\rho$ there exists a positive constant $l$ such that,
for each
$$
y\in B_{\rho}=\{y\in\mathcal{C}(\mathcal{J},\mathbb{R})
: \|y\|_{\infty}\leq \rho \},
$$
we have $\|\mathrm{T}(y)\|_{\infty}\leq l$.
By hypothesis, for each $t\in\mathcal{J}$ we have
\begin{equation*}
\begin{split}
|\mathrm{T}(y)(t)|&\leq\frac{1}{\Gamma(\alpha)}
\int_{t_{0}}^{t}(t-s)^{\alpha-1}|f(s,y(s))|\Delta s\\
&\leq\frac{M}{\Gamma(\alpha)}\int_{t_{0}}^{t}(t-s)^{\alpha-1}\Delta s\\
&\leq\frac{M}{\Gamma(\alpha)}\int_{t_{0}}^{t}(t-s)^{\alpha-1}ds\\
&\leq\frac{Ma^{\alpha}}{\alpha\Gamma(\alpha)}\\
&=\frac{Ma^{\alpha}}{\Gamma(\alpha+1)}=l.
\end{split}
\end{equation*}
\emph{Step 3:} the map $T$ sends bounded sets into equicontinuous sets of
$\mathcal{C}(\mathcal{J},\mathbb{R})$. Let $t_{1}, t_{2} \in \mathcal{J},
t_{1} < t_{2}$, $B_{\rho}$ be a bounded set of
$\mathcal{C}(\mathcal{J},\mathbb{R})$ as in Step~2,
and $y \in B_{\rho}$. Then,
\begin{equation*}
\begin{split}
|T(y)(t_{2})&-T(y)(t_{1})|\\
&\leq\frac{1}{\Gamma(\alpha)}\left|\int_{t_{0}}^{t_{1}}
(t_{1}-s)^{\alpha-1}f(s,y(s))\Delta s
-\int_{t_{0}}^{t_{2}}(t_{2}-s)^{\alpha-1}f(s,y(s))\Delta s\right|\\
&\leq\frac{1}{\Gamma(\alpha)}\Bigr|\int_{t_{0}}^{t_{1}}((t_{1}-s)^{\alpha-1}
-(t_{2}-s)^{\alpha-1} +(t_{2}-s)^{\alpha-1})f(s,y(s))\Delta s\\
&\qquad\qquad\qquad\qquad\qquad-\int_{t_{0}}^{t_{2}}(t_{2}-s)^{\alpha-1}
f(s,y(s))\Delta s\Bigr|\\
&\leq\frac{M}{\Gamma(\alpha)}\left|\int_{t_{0}}^{t_{1}}((t_{1}-s)^{\alpha-1}
-(t_{2}-s)^{\alpha-1})\Delta s+\int_{t_{1}}^{t_{2}}(t_{2}-s)^{\alpha-1}
\Delta s\right|\\
&\leq\frac{M}{\Gamma(\alpha)}\left|\int_{t_{0}}^{t_{1}}((t_{1}-s)^{\alpha-1}
-(t_{2}-s)^{\alpha-1})d s+\int_{t_{1}}^{t_{2}}(t_{2}-s)^{\alpha-1}d s\right|\\
&\leq\frac{M}{\Gamma(\alpha+1)}[(t_{2}-t_{1})^{\alpha}
+(t_{1}-t_{0})^{\alpha}-(t_{2}-t_{0})^{\alpha}]
+\frac{M}{\Gamma(\alpha+1)}(t_{2}-t_{1})^{\alpha}\\
&=\frac{2M}{\Gamma(\alpha+1)}(t_{2}-t_{1})^{\alpha}
+\frac{M}{\Gamma(\alpha+1)}[(t_{1}-t_{0})^{\alpha}-(t_{2}-t_{0})^{\alpha}].
\end{split}
\end{equation*}
As $t_{1}\rightarrow t_{2}$, the right-hand side of the above inequality tends
to zero. As a consequence of Steps~1 to 3, together with the Arzela--Ascoli
theorem, we conclude that $T : \mathcal{C}(\mathcal{J},\mathbb{R})
\rightarrow \mathcal{C}(\mathcal{J},\mathbb{R})$ is completely continuous.
\emph{Step 4:} a priori bounds. Now it remains to show that the set
$$
\Omega=\{y\in\mathcal{C}(\mathcal{J},\mathbb{R}):
y=\lambda\mathrm{T}(y), 0 <\lambda< 1\}
$$
is bounded. Let $y\in\Omega$. Then $y=\lambda\mathrm{T}(y)$
for some $0 <\lambda< 1$. Thus, for each $t\in\mathcal{J}$, we have
$$
y(t)=\lambda\left[\frac{1}{\Gamma(\alpha)}\int_{t_{0}}^{t}
(t-s)^{\alpha-1}f(s,y(s))\Delta s\right].
$$
We complete this step by considering the estimation in Step~2.
As a consequence of Schauder's fixed point theorem, we conclude
that $\mathrm{T}$ has a fixed point, which is solution
of problem~\eqref{eq1}--\eqref{eq2}.
\end{proof}


\section*{Acknowledgments}

This research is part of first author's Ph.D.,
which is carried out at Sidi Bel Abbes University, Algeria.
It was initiated while Nadia Benkhettou was visiting the
Department of Mathematics of University of Aveiro, Portugal, June of 2014.
The hospitality of the host institution and the financial
support of Sidi Bel Abbes University are here gratefully
acknowledged. Torres was supported by Portuguese funds through the
Center for Research and Development in Mathematics and Applications (CIDMA)
and the Portuguese Foundation for Science and Technology (FCT),
within project UID/MAT/04106/2013. The authors would like
to thank the Reviewers for their comments.




\begin{thebibliography}{xx}

\bibitem{AB}
R. P. Agarwal\ and\ M. Bohner,
Basic calculus on time scales and some of its applications,
Results Math. {\bf 35} (1999), no.~1-2, 3--22.

\bibitem{ABRP}
R. Agarwal, M. Bohner, D. O'Regan\ and\ A. Peterson,
Dynamic equations on time scales: a survey,
J. Comput. Appl. Math. {\bf 141} (2002), no.~1-2, 1--26.

\bibitem{AJ}
A. Ahmadkhanlu\ and\ M. Jahanshahi,
On the existence and uniqueness of solution of initial value problem
for fractional order differential equations on time scales,
Bull. Iranian Math. Soc. {\bf 38} (2012), no.~1, 241--252.

\bibitem{MR2800417}
N. R. O. Bastos, D. Mozyrska\ and\ D. F. M. Torres,
Fractional derivatives and integrals on time scales
via the inverse generalized Laplace transform,
Int. J. Math. Comput. {\bf 11} (2011), J11, 1--9.
{\tt arXiv:1012.1555}

\bibitem{BBT}
N. Benkhettou, A. M. C. Brito da Cruz\ and\ D. F. M. Torres,
A fractional calculus on arbitrary time scales: fractional
differentiation and fractional integration,
Signal Process. {\bf 107} (2015), 230--237.
{\tt arXiv:1405.2813}

\bibitem{MyID:320}
N. Benkhettou, A. M. C. Brito da Cruz\ and\ D. F. M. Torres,
Nonsymmetric and symmetric fractional calculi on arbitrary nonempty closed sets,
Math. Meth. Appl. Sci., in press, DOI: 10.1002/mma.3475
{\tt arXiv:1502.07277}

\bibitem{MyID:324}
N. Benkhettou, S. Hassani\ and\ D. F. M. Torres,
A conformable fractional calculus on arbitrary time scales,
J. King Saud Univ. Sci., in press, DOI:10.1016/j.jksus.2015.05.003
{\tt arXiv:1505.03134}

\bibitem{BP}
M. Bohner\ and\ A. Peterson,
{\it Dynamic equations on time scales},
Birkh\"auser Boston, Boston, MA, 2001.

\bibitem{BP1}
M. Bohner\ and\ A. Peterson,
{\it Advances in dynamic equations on time scales},
Birkh\"auser Boston, Boston, MA, 2003.

\bibitem{MyID:307}
M. C. Caputo\ and\ D. F. M. Torres,
Duality for the left and right fractional derivatives,
Signal Process. {\bf 107} (2015), 265--271.
{\tt arXiv:1409.5319}

\bibitem{C}
J. Cronin,
{\it Differential equations}, second edition,
Monographs and Textbooks in Pure and Applied Mathematics, 180,
Dekker, New York, 1994.

\bibitem{MyID:310}
S. Jahanshahi, E. Babolian, D. F. M. Torres\ and\ A. Vahidi,
Solving Abel integral equations of first kind via fractional calculus,
J. King Saud Univ. Sci. {\bf 27} (2015), no.~2, 161--167.
{\tt arXiv:1409.8446}

\bibitem{book:Podlubny}
I. Podlubny,
{\it Fractional differential equations},
Mathematics in Science and Engineering, 198,
Academic Press, San Diego, CA, 1999.

\end{thebibliography}
\end{document}